\documentclass[leqno,12pt]{amsart}

\theoremstyle{definition}

\usepackage[colorlinks, bookmarks=true]{hyperref}
\usepackage{color,graphicx,shortvrb}

\usepackage{enumerate}
\usepackage{amsmath}
\usepackage{amssymb}

\usepackage[latin 1]{inputenc}

\def\11{\textbf{$1$}}
\def\CC{{\mathbb{C}}}

\newcommand{\eps}{\varepsilon}
\newcommand{\norm}[1]{\left\|#1\right\|}
\newcommand{\Norm}[1]{\Bigl\|#1\Bigr\|}
\newcommand{\bgl}{\begin{eqnarray}}
\newcommand{\bglst}{\begin{eqnarray*}}
\newcommand{\egl}{\end{eqnarray}}
\newcommand{\eglst}{\end{eqnarray*}}
\newcommand{\N}{\rm{I\!N}}

\newcommand{\beinschub}{\vspace{2mm}\begin{small}\newline\makebox[5ex]{}
     \begin{minipage}[t]{75ex}}
\newcommand{\eeinschub}{\end{minipage}\end{small}\vspace{3mm}\newline}
\newcommand{\betr}[1]{| #1  |}
\newcommand{\Betr}[1]{\Bigl| #1  \Bigr|}

\newcommand{\mdE}{\,:\,}

\usepackage{amsmath, amssymb}
\usepackage{amsthm} 
\theoremstyle{plain} 
\newtheorem{theorem}{\indent\sc Theorem}[section]
\newtheorem{lemma}[theorem]{\indent\sc Lemma}
\newtheorem{corollary}[theorem]{\indent\sc Corollary}
\newtheorem{proposition}[theorem]{\indent\sc Proposition}

\theoremstyle{definition} 
\newtheorem{definition}[theorem]{\indent\sc Definition}

\begin{document}

\numberwithin{equation}{section}

\title[Weak Banach-Saks property for preduals of JBW$^*$-triples]{Weak Banach-Saks property and Koml\'os' theorem for preduals of JBW$^*$-triples}

\author[A.M. Peralta]{Antonio M. Peralta}
\address{Departamento de An{\'a}lisis Matem{\'a}tico, Universidad de Granada,\\
Facultad de Ciencias 18071, Granada, Spain}
\curraddr{Visiting Professor at Department of Mathematics, College of Science, King Saud University, P.O.Box 2455-5, Riyadh-11451, Kingdom of Saudi Arabia.}
\email{aperalta@ugr.es}

\author[H. Pfitzner]{Hermann Pfitzner}
\address{Université d'Orléans,\\
BP 6759,\\
F-45067 Orléans Cedex 2,\\
France}
\email{pfitzner@labomath.univ-orleans.fr}

\thanks{First author partially supported by the Spanish Ministry of Economy and Competitiveness project no. MTM2014-58984-P, Junta de Andaluc\'{\i}a grant FQM375 and Deanship of Scientific
Research at King Saud University (Saudi Arabia) research group no. RG-1435-020.}



\date{}
\maketitle

\begin{abstract} We show that the predual of a JBW$^*$-triple has the weak Banach-Saks property,
that is, reflexive subspaces of a JBW$^*$-triple predual are super-reflexive. We also prove that JBW$^*$-triple preduals satisfy the Koml\'os property (which can be considered an abstract
version of the weak law of large numbers).
The results rely on two previous papers from which we infer the fact that, like in the classical case of $L^1$, a subspace of
a JBW$^*$-triple predual contains $\ell_1$ as soon as it contains uniform copies of $\ell_1^n$.
\end{abstract}

\maketitle
\thispagestyle{empty}

\section{Introduction}\label{sec:intro}\noindent
From the two recent papers \cite{PePfi, PePfi2}, which precede this note, we know that if a sequence in the predual of a JBW$^*$-triple
spans the spaces $\ell_1^n$ uniformly almost isometrically then it admits a subsequence spanning $\ell_1$ almost isometrically,
as it is well known for $L^1$-spaces. This note gathers some classical consequences of this result, namely,
the weak Banach-Saks property (for which we use a criterion of Rosenthal-Beauzamy), the fact that reflexive subspaces of
JBW$^*$-triple preduals are super-reflexive, and the Koml\'os property. The latter generalizes the
(weak) law of large numbers in $L^1$ (cf.\ \cite[Thm.\ 1.9.13]{Lin-Bochner}).
Further, we prove the technical result that addition on a JBW$^*$-predual is jointly
sequentially continuous with respect to the abstract measure topology defined in \cite{Pfi15}. As a (still technical) consequence this topology is Fr\'echet-Urysohn on the unit ball.\medskip

{\it \textbf{Some notation.}}
Basic notions and properties not explained here can be found for Banach spaces in
\cite{Diestel1984, FHHMZ, JohLin_Vol1-2}, and for JBW$^*$-triples in \cite{CabRodPal2014, Chu2012},
but also, and particularly, in the introductory sections of \cite{PePfi, PePfi2}, on which this note relies heavily.
We only recall some notation concerning copies of $\ell_1$.
Let $r>0$. We say that elements $x_1,\ldots,x_n$ of a Banach space $X$ \emph{span an $r$-copy of $\ell_1^n$} if
$$\sum_{j=1}^n\betr{\alpha_j}\geq\norm{\sum_{j=1}^n \alpha_j x_j}\geq r\sum_{j=1}^n\betr{\alpha_j},$$ for all scalars $\alpha_j$.
Let $(x_j)$ be a sequence in $X$. If $(\delta_m)$ is a sequence in $[0,1[$ tending to zero and
$\displaystyle \sum_{j\geq m}\betr{\alpha_j}\geq\norm{\sum_{j\geq m} \alpha_j x_j}\geq (1-\delta_m)r\sum_{j\geq m}\betr{\alpha_j}$,
for all $m\in\N$ and all scalars $\alpha_j$
we say that $(x_j)$ \emph{spans $\ell_1$ almost $r$-isomorphically}.
When the latter holds for $r=1$ we say that $(x_j)$ \emph{spans $\ell_1$ almost isometrically}.
We say that the sequence $(x_j)$ \emph{spans the spaces $\ell_1^n$ $r$-uniformly}
(or simply \emph{spans $\ell_1^n$'s uniformly})
if for each natural $n$ there are $x_{j_1},\ldots,x_{j_n}$
spanning $\ell_1^n$ $r$-isomorphically. If  $(x_j)$ spans the spaces $\ell_1^n$ $(1-\eps)$-uniformly for each $\eps>0$ we say that $(x_j)$ \emph{spans $\ell_1^n$'s uniformly almost isometrically}.
A block sequence $(y_i)$ of $(x_n)$ is defined by $y_i=\sum_{j\in A_i}\alpha_jx_j,$ where the $A_i$ are finite pairwise disjoint subsets
of $\N$ and the $\alpha_j$ are scalars.

\section{Sequences spanning the spaces $\ell_1^n$ in the predual of a JBW$^*$-triple}

It is known that a Banach space contains the $\ell_1^n$'s uniformly almost isometrically if it contains the $\ell_1^n$'s uniformly
(see e.g. \cite[Lemma  1]{Pisier1973}). We prove the following quantitative version
only for lack of an exact reference.

\begin{lemma}\label{l James distortion finite}
If a sequence $(x_j)$ of the closed unit ball of a Banach space spans $\ell_1^n$'s $r$-uniformly {\rm(}where $r>0${\rm)},
then there is a block sequence $(y_i)$ spanning $\ell_1^n$'s uniformly almost isometrically;
more specifically, the blocks $\displaystyle y_i=\sum_{j\in A_i}\alpha_jx_j$ are such that for each $n\in\N$ the finite sequence $(y_i)_{i\in B_n}$,
where $\displaystyle B_n=\left\{1+\frac{(n-1)n}{2},\ldots,\frac{n(n+1)}{2}\right\}$, spans a $(1-2^{-n})$-isomorphic copy of $\ell_1^n$,
$\norm{y_i}=1,$ and $\displaystyle \sum_{j\in A_i} \betr{\alpha_j}<\frac{1+2^{-i}}{r},$ for all $i\in\N.$
\end{lemma}

\begin{proof}
Suppose that $(x_j)$ spans $\ell_1^n$'s $r$-uniformly ($r>0$).
Since the $B_n$ (as in the statement of the lemma) form a partition of ${\N}=\bigcup_{n\in\N}B_n$ we construct $(y_i)$ by
constructing  $(y_i)_{i\in B_n}$ inductively over $n$. Let $n\in\N$, suppose $(y_i)_{i\in B_1}$, \ldots,  $(y_i)_{i\in B_{n-1}}$
have been constructed and set $p=\max\bigcup_{q=1}^{n-1}\bigcup_{i\in B_q}A_i$ (with $p=0$ if $n=1$).
Then the sequence $(x_{j+p})$ spans $\ell_1^n$'s $r$-uniformly, too.

%
Set $\rho={r}/{(1+2^{-{n(n+1)}/{2}})}< r$.
For each natural $m$ we define
$$ r_m :=\sup_{x_{j_1}, \ldots, x_{j_m}\geq p+1}\;\inf_{\sum_{k=1}^m\betr{\alpha_k}=1}\norm{\sum_{k=1}^m \alpha_k x_{j_k}}.$$
Then $r_m\geq r$ for all $m$.
The sequence $(r_m)$ is decreasing because for any $\gamma>0$ and any $m$ there are $x_{j_1}, \ldots, x_{j_{m+1}}\geq p+1$ such that
$\displaystyle r_{m+1}-\gamma<\inf_{\sum_{k=1}^{m+1}\betr{\alpha_k}=1}\norm{\sum_{k=1}^{m+1} \alpha_k x_{j_k}}
\leq r_m$.\smallskip

Set $r'=\lim_{m} r_m$, then we have $r'\geq r$.
Choose $\eta>0$ such that $\frac{r'-\eta}{r'+\eta}>1-2^{-n}$ and $r-\eta>\rho$
and choose $m_1$ such that $r_{m_1}<r'+\eta$. Set $m_2=nm_1$.
Finally, choose $x_{j_1}, \ldots, x_{j_{m_2}}$ such that
$\displaystyle r_{m_2}-\eta<\inf_{\sum_{k=1}^{m_2}\betr{\alpha_k}=1}\norm{\sum_{k=1}^{m_2} \alpha_k x_{j_k}}$.

Set $A'_l=\left\{j_k\mdE k\in\{1+(l-1)m_1, \ldots, lm_1\}\right\},$ for $l=1, \ldots, n$. The $A'_l$ are pairwise disjoint and are disjoint
from the $A_i$ constructed so far (i.e. $i\in\bigcup_{q=1}^{n-1}B_q$ (if $n\geq2$)).
If we set $\displaystyle \tilde{z}_l=\sum_{j\in A'_l}\alpha_j^{(l)}x_j$, for $l\leq n$, where the $\alpha_j^{(l)}$ are such that
$\displaystyle \sum_{j\in A'_l}\betr{\alpha_j^{(l)}}=1$ and
$\displaystyle \norm{\tilde{z}_l}=\inf_{\sum_{j\in A'_l}\betr{\alpha_j}=1}\norm{\sum_{j\in A'_l} \alpha_j x_j},$ then $\norm{\tilde{z}_l}\leq r_{m_1}<r'+\eta$.\smallskip

Further, we have
$$\inf_{\sum_{l=1}^n\betr{\beta_l}=1}\norm{\sum_{l=1}^n \beta_l \tilde{z}_l}
\geq\inf_{\sum_{k=1}^{m_2}\betr{\alpha_k}=1}\norm{\sum_{k=1}^{m_2} \alpha_k x_{j_k}}>r_{m_2}-\eta\geq r'-\eta.$$
In particular, $\norm{\tilde{z}_l}>r'-\eta>\rho$.
Now we set $z_l=\displaystyle\frac{\tilde{z}_l}{\norm{\tilde{z}_l}},$ hence $\norm{z_l}=1$ and
$\displaystyle z_l=\sum_{j\in A'_l}\frac{\alpha_j^{(l)}}{\norm{\tilde{z}_l}} x_j$ with
$\displaystyle \sum_{j\in A'_l}\frac{\betr{\alpha_j^{(l)}}}{\norm{\tilde{z}_l}} < 1/\rho=(1+2^{-n(n+1)/2})/r \leq
(1+2^{-i})/r,$ for all $i\in B_n$.
Moreover, for all scalars $\beta_l$,
$$\sum_{l=1}^n\betr{\beta_l}\geq\norm{\sum_{l=1}^n\beta_l z_l}\geq\frac{r'-\eta}{r'+\eta}\sum_{l=1}^n\betr{\beta_l}
\geq(1-2^{-n})\sum_{l=1}^m\betr{\beta_l},$$
which shows that the $z_1, \ldots, z_n$ span a $(1-2^{-n})$-isomorphic copy of $\ell_1^n$.

It remains to set $A_i=A'_{i-((n-1)n/2)}$ and $y_i=z_{i-((n-1)n/2)}$ for all $i\in B_n$.
This ends the induction and the proof.
\end{proof}

The following proposition is a classical result for $L^1$ (see e.g. \cite{Rosenth1973}) and follows from \cite[Thm.\ 3.1]{RaynaudXu}
for preduals of von Neumann algebras.

\begin{proposition}\label{t l1 uniformly hence isomorphically}
If a sequence in the predual of a JBW$^*$-triple spans $\ell_1^n$'s $r$-uniformly then it admits a subsequence spanning
$\ell_1$ almost $r$-isomorphically.
\end{proposition}

\begin{proof}
Let us first observe a consequence of an extraction lemma of Simons \cite{Sim-DPP}.
Given $\eps>0$, a bounded sequence $(x_l)$ in a Banach space $X$ and a weakly null sequence $(x_l^*)$ in the dual $X^*$ there is a sequence
$(l_n)$ in $\N$ such that $\displaystyle \sum_{k\neq n}\betr{x_{l_k}^*(x_{l_n})}<\eps$ for all $n$. If we apply this successively for
$\eps=2^{-m}$, $m=1, 2, \ldots$ and pass to the diagonal sequence then we get
$\displaystyle \sum_{k\geq m,\;k\neq n}\betr{x_{l_k}^*(x_{l_n})}<2^{-m}$ for all $m$ and $n\geq m$.\smallskip

Suppose $(\psi_j)$ spans $\ell_1^n$'s $r$-uniformly in the predual, $W_*$, of a JBW$^*$-triple $W$.
We may suppose that $\norm{\psi_j}\leq1$ for all $j$.
Take blocks $\displaystyle \varphi_i=\sum_{j\in A_i}\alpha_j\psi_j$ with
\bgl
\sum_{j\in A_i}\betr{\alpha_j}<\frac{1+2^{-i}}{r} \label{eq 221}
\egl as given by Lemma \ref{l James distortion finite}.
By \cite[Thm.\ 4.1]{PePfi} there are a subsequence $(\varphi_{i_n})$
and a sequence $(\widetilde{\varphi}_n)$ of pairwise orthogonal functionals in $W_*$ such that $\|\varphi_{i_n} -  \widetilde{\varphi}_n\|<2^{-n}$.
Let $u_n$ be the support tripotent of $\widetilde{\varphi}_n$ in $W$.
Then $u_n(\widetilde{\varphi}_k)= \delta_{k,n}$ $\forall n, k\in \mathbb{N}$.
Moreover the $u_n$ are pairwise orthogonal, and thus $\norm{\sum\theta_nu_n}\leq1$ if $\betr{\theta_n}\leq1$ for all $n$, and $(u_n)\to 0$ weakly in $W$.
By \eqref{eq 221}, and since $\betr{u_n(\varphi_{i_n})}>1-2^{-n}$, we obtain $j_n\in A_{i_n}$ such that
\bglst
\betr{u_n(\psi_{j_n})}>\frac{1-2^{-n}}{1+2^{-n}}r
\eglst
(because otherwise we would have
$\displaystyle \betr{u_n(\varphi_{i_n})}
\leq\sum_{j\in A_{i_n}}\betr{\alpha_j}\,\betr{u_n(\psi_j)}\leq\frac{1+2^{-i_n}}{r}\;\frac{1-2^{-n}}{1+2^{-n}}r\leq1-2^{-n}\mbox{).}$
By the observation above we may suppose (by passing to appropriate subsequences) that
$\displaystyle \sum_{k\geq m,\;k\neq n}\betr{u_k(\psi_{j_n})}<2^{-m}$ for all $m$ and $n\geq m$.\smallskip

Fix $m\in\N$. Let $\beta_n$, $\theta_n$ be scalars such that $\betr{\theta_n}=1$ and
$\theta_nu_n(\beta_n\psi_{j_n})=\betr{\beta_n}\,\betr{u_n(\psi_{j_n})}$.
We have
\bglst
\sum_{n\geq m}\betr{\beta_n}
\geq\Norm{\sum_{n\geq m}\beta_n\psi_{j_n}}
&\geq&\Betr{\big(\sum_{k\geq m}\theta_ku_k\big)\big(\sum_{n\geq m}\beta_n\psi_{j_n}\big)}\\
&\geq&\sum_{n\geq m}\betr{\beta_n}\,\betr{u_n(\psi_{j_n})}-\sum_{n\geq m}\sum_{k\geq m,\;k\neq n}\betr{\beta_n}\,\betr{u_k(\psi_{j_n})}\\
&\geq&\frac{1-2^{-m}}{1+2^{-m}}r\sum_{n\geq m}\betr{\beta_n}-2^{-m} \ \sum_{n\geq m}\betr{\beta_n}\,
\geq
(1-\delta_m)r\sum_{n\geq m}\betr{\beta_n}
\eglst
where $0<\delta_m=1-(\frac{1-2^{-m}}{1+2^{-m}}-\frac{2^{-m}}{r})\to0$ as $m\to\infty$.
\end{proof}

We briefly recall that a Banach space is said to be \emph{super-reflexive} if all its
ultrapowers are reflexive.

\begin{corollary}\label{c superrefl}
Reflexive subspaces of the predual of a JBW$^*$-triple are super-reflexive.
\end{corollary}

\begin{proof}
Let $W_*$ be the predual of a JBW$^*$-triple $W$, let $U$ be a closed subspace of $W_*$.
Suppose there is an ultrapower $(U)_{\mathcal{U}}$ which is not reflexive.
This ultrapower is a subspace of the ultrapower $(W_*)_{\mathcal{U}}$ and the
latter is the predual of a JBW$^*$-triple $\mathcal{W}$ (\cite[Prop.\ 5.5]{BeMar05}).
Hence, by weak sequential completeness of $\mathcal{W}_*$ and Rosenthal's $\ell_1$-theorem,  $(U)_{\mathcal{U}}$ contains an isomorphic copy of $\ell_1$.
This copy is finitely representable in $U$ (\cite[Thm.\ 6.3]{Hein80}) hence, by Proposition \ref{t l1 uniformly hence isomorphically},
$U$ contains $\ell_1$ and is not reflexive.
\end{proof}

The above result was established by Jarchow \cite{Jarchow1986} in the particular setting of subspaces of the predual of a von Neumann algebra.

\begin{definition}
A Banach space is said to have the \emph{Banach-Saks property} if every bounded sequence $(x_n)$
admits a subsequence $(x_{n_k})$ such that the Ces\`aro means $\frac{1}{N}\sum_{k=1}^Nx_{n_k}$
(equivalently, the Ces\`aro means $\frac{1}{N}\sum_{k=1}^Ny_k$ of any further subsequence $(y_k)$ of $(x_{n_k})$) converge in norm.


A Banach space is said to have the \emph{weak Banach-Saks property} (or the Banach-Saks-Rosenthal property) if every weakly null sequence $(x_n)$
admits a subsequence $(x_{n_k})$ such that the Ces\`aro means $\frac{1}{N}\sum_{k=1}^Nx_{n_k}$
(equivalently, the Ces\`aro means $\frac{1}{N}\sum_{k=1}^Ny_k$ of any further subsequence $(y_k)$ of $(x_{n_k})$)
converge in norm.
\end{definition}
The equivalences in the definitions come from a result of Erd\"os and Magidor \cite{ErdosMagidor} (or \cite[Prop.\ II.6.1]{BeauLa}), according to which any bounded sequence in a Banach space admits a subsequence such that either all of its subsequences have norm convergent Ces\`aro means or none of its subsequences has norm convergent Ces\`aro means.
The Banach-Saks property implies reflexivity \cite[p.\ 38]{BeauLa} (but not conversely \cite{Baernstein}, \cite[V]{Beauz})
and is implied by superreflexivity \cite{Kakutani1938}. Therefore, by Corollary \ref{c superrefl}, a subspace of a JBW$^*$-predual
is reflexive if, and only if, it is super-reflexive if, and only if, it has the (super-)Banach-Saks property.
For detailed information on the Banach-Saks property and its variants we refer to \cite{Beauz, BeauLa}.\smallskip

It is due to Szlenk that every weakly convergent sequence in $L_1[0,1]$ has a subsequence whose arithmetic means are norm convergent
(cf. \cite{Szlenk1965}, and \cite[p.\ 112]{Diestel1984}).
Szlenk's result on the weak Banach-Saks property has been extended in \cite{BelDie} to duals of C$^*$-algebras,
and implicitly to preduals of von Neumann algebras (see the comment in MR0848901).
The explicit statement for the predual of a von Neumann algebra $N$ appears in \cite[Remark in \S5]{RaynaudXu};
cf.\ also \cite[Theorem 5.4]{HaagRosSuk} for the case in which $N$ is a finite von Neumann algebra.
Next, we establish Szlenk's theorem for JBW$^*$-triple preduals.

\begin{corollary}\label{c wBS}
The predual of a JBW$^*$-triple has the weak Banach-Saks property.
\end{corollary}
\begin{proof}
If a Banach space fails the weak Banach-Saks property then, by a result of Rosenthal (cf.\ \cite[Prop.\ II.1]{Beauz} or \cite{BeauLa})
it contains a weakly null sequence which has property $(\mathcal{P}_2)$ in the sense of \cite{Beauz, BeauLa},
and therefore spans $\ell_1^n$'s uniformly.
By Proposition \ref{t l1 uniformly hence isomorphically} this sequence admits a subsequence spanning an isomorphic copy of $\ell_1$
which is not possible for a weakly convergent sequence.
\end{proof}

\section{The abstract measure topology on the predual of a JBW$^*$-triple}

A Banach space $X$ is called L-embedded if its bidual can be written $X^{**}=X\oplus_1 X_s$.
See the standard reference \cite{HWW} for details on L-embedded spaces.
Preduals of von Neumann algebras and JBW$^*$-triples are L-embedded spaces (cf. \cite[Example IV.1.1]{HWW} and \cite[Proposition 3.4]{BartTim}).\smallskip

In \cite[\S 5]{Pfi15} a topology, called the \emph{abstract measure topology}, has been defined on L-embedded Banach spaces $X$, with the particularity that if $X$ is a commutative or a non-commutative $L^1$ with finite measure/trace, this topology coincides on (norm) bounded sets with the usual measure topology. The abstract measure topology of an L-embedded space $X$ will be denoted by $\tau_\mu$. We recall some properties of $\tau_\mu$. It is a sequential topology, which means that sets are closed if they are sequentially closed. Therefore $\tau_\mu$ is determined by its convergent sequences: a sequence $(x_n)$ converges to $x$ in $\tau_\mu$ if and only if $(x_n-x)$ is $\tau_\mu$-null and a sequence $(y_n)$ is $\tau_\mu$-null if and only if each subsequence $(y_{n_k})$ contains a further subsequence $(y_{n_{k_m}})$ which either is norm null or is semi-normalized (i.e. bounded and $\inf\norm{y_{n_{k_m}}}>0$) and such that $(y_{n_{k_m}}/\norm{y_{n_{k_m}}})$ spans $\ell_1$ almost isometrically. In particular,
 $\tau_\mu$-convergent sequences are (norm) bounded.
(Note that in \cite{Pfi15} the definition of 'to span $\ell_1$ almost isometrically' differs slightly from ours but coincides with ours for normalized sequences.)
It is not known whether $\tau_\mu$ is Hausdorff but convergent sequences have unique limits (cf. \cite[Proof of Theorem 5.2]{Pfi15}).
By definition addition is separately continuous with respect to $\tau_\mu$, but it is an open problem whether addition is
sequentially jointly continuous or even jointly continuous on an L-embedded space $X$ (cf. \cite[Question 2]{Pfi15}). It can be shown from the results in \cite{RaynaudXu}, that addition is  $\tau_\mu$-sequentially continuous on the unit ball of a von Neumann algebra predual (see \cite[Comments in Question 2]{Pfi15}). In this section we shall enlarge the examples list by showing that a similar statement remains true for JBW$^*$-triple preduals.\smallskip

Bounded sets of an L-embedded space satisfy a mild form of sequential compactness with respect to $\tau_\mu$: a bounded sequence
admits convex blocks that $\tau_\mu$-converge \cite[Thm.\ 6.1]{Pfi15}.
We say that an L-embedded Banach space has the Koml\'os property if each bounded sequence admits a subsequence such that the
Ces\`aro means of any further subsequence converge with respect to $\tau_\mu$ (to the same limit). This property
is strictly stronger than the just described convex-block compactness because there are L-embedded Banach spaces failing it
\cite[Example 6.2]{Pfi15}. Koml\'os shows in \cite{Komlos} that $L^1[0,1]$ has this property
and so do preduals of JBW$^*$-triple preduals as we shall show next.

\begin{theorem}
The predual of a JBW$^*$-triple has the Koml\'os property.
\end{theorem}

\begin{proof}
The theorem is immediate from the splitting property shown in \cite[Thm.\ 6.1]{PePfi2}, from the weak Banach-Saks property Corollary \ref{c wBS}
and from \cite[\S6, Observation]{Pfi15} or \cite[p.\ 637]{Randri_2003}.
\end{proof}

We can now explore the $\tau_\mu$-sequential continuity of the addition on a JBW$^*$-triple predual.

\begin{theorem}
Let $W$ be a JBW$^*$-triple with predual $W_*$. Let $W_*$ be endowed with its abstract measure topology $\tau_\mu$.
Then addition is sequentially $\tau_\mu$-continuous. Consequently, $\tau_\mu$ is Fr\'echet-Urysohn on the closed unit ball of $W_*$.
\end{theorem}

\begin{proof}
Let $\varphi_n\stackrel{\tau_\mu}{\to}\varphi$ and $\psi_n\stackrel{\tau_\mu}{\to}\psi$ in $W_*$.
Since $\tau_\mu$ is a topology, in order to show that addition is sequentially $\tau_\mu$-continuous it is enough
to show that $(\varphi_n+\psi_n)$ contains a subsequence
$\tau_\mu$-converging to $\varphi+\psi$. By translation invariance of $\tau_\mu$ we may and do suppose that $\varphi=\phi=0$.
By definition of $\tau_\mu$-convergent sequences it remains to show that $(\varphi_n+\psi_n)$ contains a subsequence which
either is norm null or is semi-normalized and such that $((\varphi_n+\psi_n)/\norm{\varphi_n+\psi_n})$ spans $\ell_1$ almost isometrically.
If at least one of the two sequences $(\varphi_n)$, $(\psi_n)$ contains a norm null subsequence then we are done \cite[\S5 Remark (b), page 433]{Pfi15}.
Therefore, the only interesting case we are left with is the following: all $\varphi_n, \psi_n, \varphi_n+\psi_n$ are $\neq0$, the limits
$\lim\norm{\varphi_n}$, $\lim\norm{\psi_n}$ and $\lim\norm{\varphi_n+\psi_n}$ exist and are $>0$ and
both $(\varphi_n/\norm{\varphi_n})$ and $(\psi_n/\norm{\psi_n})$ span $\ell_1$ almost isometrically.\smallskip

In order to show that there is a subsequence of $((\varphi_n+\psi_n)/\norm{\varphi_n+\psi_n})$ which spans $\ell_1$ almost isometrically
it is enough, by \cite[Thm.\ 4.1]{PePfi}, to show that $((\varphi_n+\psi_n)/\norm{\varphi_n+\psi_n})$ spans $\ell_1^k$'s uniformly, that is,
it is enough to show for each $k\in\N$ that if for each $\eps'>0$ some of the $(\varphi_n+\psi_n)/\norm{\varphi_n+\psi_n}$ span
a $(1-\eps')$-copy of $\ell_1^k$ then for each $\eps>0$ some finite selection from $((\varphi_n+\psi_n)/\norm{\varphi_n+\psi_n})$ spans a $(1-\eps)$-copy of $\ell_{1}^{k+1}$.\smallskip

Let $k\in\N$ and $\eps>0$. Choose $\eps'>0$ such that $(1-\eps')^2>1-\eps$.
Take $\varphi_{n_1}+\psi_{n_1},\ldots,\varphi_{n_k}+\psi_{n_k}$ such that
\bglst
\sum_{l=1}^k\betr{\alpha_l}\norm{\varphi_{n_l}+\psi_{n_l}}
\geq\Norm{\sum_{l=1}^k\alpha_l(\varphi_{n_l}+\psi_{n_l})}
\geq(1-\eps')\sum_{l=1}^k\betr{\alpha_l}\norm{\varphi_{n_l}+\psi_{n_l}}
\eglst
for all scalars $\alpha_l$.\smallskip

The following claim is a consequence of Godefroy's well known way to use the $w^*$-lower semicontinuity of the
norm in L-embedded Banach spaces (\cite{G-bien}, \cite[IV.2.2]{HWW}) (as it has been used, for example, throughout \cite{Pfi15}).
The proof is inserted here for the reader's convenience.\smallskip

\emph{Claim.} Let $F\subset X$ be a finite dimensional subspace of an L-embedded Banach space $X$ and let $(x_n)$ be a sequence such that
$(x_n/\norm{x_n})$ spans $\ell_1$ almost isometrically in $X$ and $\lim_{\mathcal{V}}\norm{x_n}>0$ where $\mathcal{V}$ is a
non trivial ultrafilter on $\N$.\\
Then, for every $\eta>0$, there is $U\in\mathcal{V}$ such that
\bgl
\norm{\alpha x+\beta x_n}\geq(1-\eta)(\betr{\alpha }\norm{x}+\betr{\beta}\norm{x_n})
\quad\forall \alpha , \beta\in\CC,\;\forall n\in U,\;\forall x\in F.\label{eq 322}
\egl
\emph{Proof of the claim.} Let $x_s$ be a $w^*$-accumulation point of $(x_n)$ along $\mathcal{V}$.
By \cite[Lemma  3.2]{Pfi15}, $x_s$ satisfies $x_s\in X_s$ (where $X_s\subset X^{**}$ is as above)
and $\norm{x_s}=\lim_{\mathcal{V}}\norm{x_n}$ hence
$\norm{\alpha x+\beta x_s}=\betr{\alpha }\norm{x}+\betr{\beta}\norm{x_s}$
for all $\alpha , \beta\in\CC$, $x\in X$.

Given a fixed $x\in F$, $\norm{x}=1$ and finitely many $(\alpha,\beta)\in\ell_1^2$,
from $\liminf\norm{\alpha x+\beta\frac{x_n}{\norm{x_n}}}\geq\norm{\alpha x+\beta\frac{x_s}{\norm{x_s}}}$, we deduce the existence of $U\in \mathcal{V}$ such that
\bgl
\norm{\alpha x+\beta\frac{x_n}{\norm{x_n}}}\geq(1-\eta/3)(\betr{\alpha}+\betr{\beta}), \label{eq 323}
\egl
for all the $(\alpha,\beta)$ chosen above and $n\in U$. If these finitely many $(\alpha,\beta)$ are chosen to form an $\eta/3$-net of the unit sphere of $\ell_1^2$, then we get \eqref{eq 323} for all $\alpha,\beta\in\CC$ and $n\in U$, but with $2\eta/3$ instead of $\eta/3$.
Finally, if one repeats the same reasoning for finitely many $x$ which form an $\eta/3$-net of the unit sphere of $F$, then the conclusion of the claim follows.\smallskip

Let $F$ be the $k$-dimensional subspace of $W_*$ spanned by the $\varphi_{n_1}+\psi_{n_1},\ldots,\varphi_{n_k}+\psi_{n_k}$.
We apply the claim to $X=W_*$ with $\eta=1/m$, $m\in\N$, which gives a subsequence $(\varphi_{n_m})$ such that
$\norm{\alpha \xi+\beta \varphi_{n_m}}\geq(1-1/m)(\betr{\alpha }\norm{\xi}+\betr{\beta}\norm{\varphi_{n_m}})$
for all $\alpha , \beta\in\CC$, $\xi\in F$.
By the same argument applied to the corresponding subsequence of $(\psi_n)$ we may and will assume that also
$\norm{\alpha \xi+\beta \psi_{n_m}}\geq(1-1/m)(\betr{\alpha }\norm{\xi}+\betr{\beta}\norm{\psi_{n_m}})$
for all $\alpha , \beta\in\CC$, $\xi\in F$.\smallskip

Now we consider $\widehat{\xi}=[\xi]_{\mathcal{U}}$ for $\xi\in F$, $\widetilde{\varphi}=[\varphi_{n_m}]_{\mathcal{U}}$
and $\widetilde{\psi}=[\psi_{n_m}]_{\mathcal{U}}$ in the ultrapower $(W_*)_{\mathcal{U}}$ where ${\mathcal{U}}$ is non trivial
ultrafilter on $\N$. It is known that $(W_*)_{\mathcal{U}}$ is the predual of a JBW$^*$-triple $\mathcal{W}$ (cf. \cite[Prop.\ 5.5]{BeMar05}).
The functionals $\widehat{\xi}$ and $\widetilde{\varphi}$ are L-orthogonal ($\widehat{\xi}\perp_{L} \widetilde{\varphi}$) in $\mathcal{W}_*=(W_*)_{\mathcal{U}}$ because
$$\norm{\alpha\widehat{\xi}+\beta\widetilde{\varphi}}=\lim_{\mathcal{U}}\norm{\alpha\xi+\beta\varphi_{n_m}}
=\betr{\alpha}\,\|\widehat{\xi}\|+\betr{\beta}\,\norm{\widetilde{\varphi}}$$
for all $\xi\in F$. Likewise, $\widehat{\xi}\perp_{L} \widetilde{\psi}$. It is known that $\widehat{\xi}\perp_{L} \widetilde{\varphi}$  if and only if $\widehat{\xi}\perp \widetilde{\varphi}$, that is, the support tripotents of $\widehat{\xi}$ and $\widetilde{\varphi}$ in $\mathcal{W}$ are orthogonal (i.e. $s(\widehat{\xi})\perp s(\widetilde{\varphi})$)
(cf. \cite[Lemma 2.3]{FriRu87}). Similarly we get $s(\widehat{\xi})\perp_{L} s(\widetilde{\psi})$. Thus $s(\widetilde{\varphi}), s(\widetilde{\psi})\in \mathcal{W}_0(s(\widehat{\xi}))$, and hence $\widehat{\xi}\perp \widetilde{\varphi}+\widetilde{\psi}$ (because $\widetilde{\varphi}$ and $\widetilde{\psi}$ lie in $({\mathcal{W}}_*)_0(s(\widehat{\xi}))$). We deduce that $\widehat{\xi}\perp_{L} \widetilde{\varphi}+\widetilde{\psi}$ and
$$\lim_{\mathcal{U}}\norm{\xi+\beta(\varphi_{n_m}+\psi_{n_m})}=\norm{\widehat{\xi}+\beta(\widetilde{\varphi}+\widetilde{\psi})}
=\norm{\xi}+\betr{\beta}\lim_{\mathcal{U}}\norm{\varphi_{n_m}+\psi_{n_m}}.$$
Similarly as above we see that there is $n_{k+1}> n_k$ such that for all $\beta\in\CC$, $\xi\in F$ we have
$$\lim_{\mathcal{U}}\norm{\xi+\beta(\varphi_{n_{m}}+\psi_{n_{m}})}
\geq(1-\eps')\big(\norm{\xi}+\betr{\beta}\norm{\varphi_{n_{k+1}}+\psi_{n_{k+1}}}\big).$$
Let $\alpha_1,\ldots,\alpha_{k+1}\in\CC$ and consider $\xi=\sum_{l=1}^{k}\alpha_l(\varphi_{n_l}+\psi_{n_l})\in F$.
Then
\bglst
\norm{\sum_{l=1}^{k+1}\alpha_l(\varphi_{n_l}+\psi_{n_l})}
&=&\norm{\xi+\alpha_{k+1}(\varphi_{n_{k+1}}+\psi_{n_{k+1}})}\\
&\geq&(1-\eps')\left(\big((1-\eps')\sum_{l=1}^{k}\betr{\alpha_l}\norm{\varphi_{n_l}+\psi_{n_l}}\big)+\betr{\alpha_{k+1}}\norm{\varphi_{n_{k+1}}+\psi_{n_{k+1}}}\right)\\
&\geq&(1-\eps)\sum_{l=1}^{k+1}\betr{\alpha_l}\norm{\varphi_{n_m}+\psi_{n_m}}.
\eglst
The last assertion follows from \cite[Lemma  5.3]{Pfi15}.
\end{proof}

\begin{thebibliography}{10}

\bibitem{Baernstein}
A.~Baernstein.
\newblock On reflexivity and summability.
\newblock {\em Studia Math.}, 42:91--94, 1972.

\bibitem{BartTim} T.~Barton and R.~Timoney.
\newblock Weak$^*$-continuity of Jordan triple products and its applications.
\newblock {\em Math. Scand.}, 59(2):177--191, 1986.

\bibitem{Beauz}
B.~Beauzamy.
\newblock {Banach-Saks properties and spreading models}.
\newblock {\em Math. Scand.}, 44:357--384, 1979.

\bibitem{BeauLa}
B.~Beauzamy and J.-T. Laprest\'{e}.
\newblock {\em Mod\`{e}les \'{e}tal\'{e}s des epaces de Banach}.
\newblock Travaux en Cours. Hermann, Paris, 1984.

\bibitem{BeMar05}
J. {Becerra Guerrero} and M. {Mart\'{\i}n}.
\newblock {The Daugavet property of $C^{\ast}$-algebras, $JB^{\ast}$-triples,
  and of their isometric preduals.}
\newblock {\em {J. Funct. Anal.}}, 224(2):316--337, 2005.

\bibitem{BelDie}
A.~Belanger and J.~Diestel.
\newblock {A remark on weak convergence in the dual of a
  C{$^{*}$}-al\-ge\-bra}.
\newblock {\em Proc. Amer. Math. Soc.}, 98:185--186, 1986.

\bibitem{CabRodPal2014}
M.~Cabrera and A.~Rodriguez Palacios.
\newblock {\em {Non-Associative Normed Algebras: Volume 1, The Vidav-Palmer and
  Gelfand-Naimark Theorems.}}
\newblock Cambridge University Press, 2014.

\bibitem{Chu2012}
Ch.-H.~{Chu}.
\newblock {\em {Jordan structures in geometry and analysis.}}
\newblock Cambridge: Cambridge University Press, 2012.

\bibitem{Diestel1984}
J.~{Diestel}.
\newblock {Sequences and series in a Banach space.}
\newblock {Graduate Texts in Mathematics, 92. New York-Heidelberg-Berlin:
  Springer-Verlag. XIII, 261 p. DM 108.00 (1984).}, 1984.

\bibitem{ErdosMagidor}
P.~{Erd\H{o}s} and M.~{Magidor}.
\newblock {A note on regular methods of summability and the Banach-Saks
  property.}
\newblock {\em {Proc. Am. Math. Soc.}}, 59:232--234, 1976.

\bibitem{FHHMZ}
M.~{Fabian}, P.~{Habala}, P.~{H\'ajek}, V.~{Montesinos}, and
  V.~{Zizler}.
\newblock {\em {Banach space theory. The basis for linear and nonlinear
  analysis.}}
\newblock Berlin: Springer, 2011.

\bibitem{FriRu87}
Y.~{Friedman} and B.~{Russo}.
\newblock {Conditional expectation and bicontractive projections on Jordan
  $C\sp *$- algebras and their generalizations.}
\newblock {\em {Math. Z.}}, 194:227--236, 1987.

\bibitem{G-bien}
G.~Godefroy.
\newblock Sous-espaces bien dispos\'{e}s de {$L^{1}$} -- {Applications}.
\newblock {\em Trans. Amer. Math. Soc.}, 286:227--249, 1984.

\bibitem{HaagRosSuk}
U.~{Haagerup}, H.~P. {Rosenthal}, and F.~A. {Sukochev}.
\newblock {Banach embedding properties of non-commutative $L^p$-spaces.}
\newblock {\em {Mem. Am. Math. Soc.}}, 776:68, 2003.

\bibitem{HWW}
P.~{Harmand}, D.~{Werner}, and W.~{Werner}.
\newblock {\em {$M$-ideals in Banach spaces and Banach algebras.}}
\newblock Berlin: Springer-Verlag, 1993.

\bibitem{Hein80}
S.~{Heinrich}.
\newblock {Ultraproducts in Banach space theory.}
\newblock {\em {J. Reine Angew. Math.}}, 313:72--104, 1980.

\bibitem{Jarchow1986}
Hans {Jarchow}.
\newblock {On weakly compact operators on $C\sp*$-algebras.}
\newblock {\em {Math. Ann.}}, 273:341--343, 1986.

\bibitem{JohLin_Vol1-2}
W.B.~{Johnson} and J.~{Lindenstrauss}, editors.
\newblock {\em {Handbook of the geometry of Banach spaces. Volumes 1. and 2.}}
\newblock Amsterdam: Elsevier, 2001, 2003.

\bibitem{Kakutani1938}
S.~Kakutani.
\newblock Weak convergence in uniformly convex spaces.
\newblock {\em Tohoku Math. J.}, 45:188--193, 1938.

\bibitem{Komlos}
J.~Koml\'{o}s.
\newblock A generalization of a problem of {Steinhaus}.
\newblock {\em Acta Math. Acad. Sci. Hung.}, 18:217--229, 1967.

\bibitem{Lin-Bochner}
P.~K.~Lin.
\newblock {\em K{\"{o}}the-{Bochner} function spaces}.
\newblock Birkh{\"{a}}user, Boston, 2004.

\bibitem{PePfi}
A.M.~Peralta and H.~Pfitzner.
\newblock Perturbation of $\ell_1$-copies in preduals of {JBW$^*$-triples}.
\newblock Preprint 2014.

\bibitem{PePfi2}
A.M.~Peralta and H.~Pfitzner.
\newblock {The Ka\-dec-{Pe\l{}\-czy\'ns\-ki}-Ro\-sen\-thal subsequence
  splitting lemma for JBW$^*$-triple preduals}.
\newblock to appear in \emph{Studia Math.}

\bibitem{Pfi15}
H.~Pfitzner.
\newblock {$L$-embedded Banach spaces and measure topology}.
\newblock {\em Israel J. Math.}, 205(1):421--451, 2015.

\bibitem{Pisier1973}
G.~{Pisier}.
\newblock {Sur les espaces de Banach qui ne contiennent pas uniformement de
  $l^1_n$.}
\newblock {\em {C. R. Acad. Sci., Paris, S\'er. A}}, 277:991--994, 1973.

\bibitem{Randri_2003}
N.~Randrianantoanina.
\newblock {Non-commutative subsequence principles}.
\newblock {\em Math. Z.}, 245:625--644, 2003.

\bibitem{RaynaudXu}
Y.~{Raynaud} and Q.~{Xu}.
\newblock {On subspaces of non-commutative $L_{p}$-spaces.}
\newblock {\em {J. Funct. Anal.}}, 203(1):149--196, 2003.

\bibitem{Rosenth1973}
H.~P. Rosenthal.
\newblock On subspaces of {$L^{p}$}.
\newblock {\em Ann. of Math.}, 97:344--373, 1973.

\bibitem{Sim-DPP}
S.~{Simons}.
\newblock {On the Dunford-Pettis property and Banach spaces that contain
  $c_0$.}
\newblock {\em {Math. Ann.}}, 216:225--231, 1975.

\bibitem{Szlenk1965}
W.~{Szlenk}.
\newblock {Sur les suites faiblement convergentes dans l'espace L.}
\newblock {\em {Stud. Math.}}, 25:337--341, 1965.

\end{thebibliography}

\end{document}